\documentclass[12pt]{amsart}
\usepackage{amssymb,latexsym,eufrak,amsmath,amscd, graphicx}
\usepackage[dvipsnames,usenames]{color}
\usepackage[colorlinks=true,pagebackref,hyperindex]{hyperref}   
\hypersetup{ linkcolor=RawSienna, anchorcolor=BurntOrange,
citecolor=OliveGreen, filecolor=BlueViolet, menucolor=Yellow,
urlcolor=OliveGreen }

\usepackage{amsfonts}
\usepackage[all]{xypic}
\UseTips

\renewcommand{\to}[1][]{\xrightarrow{\ #1\ }}

\newcommand{\forget}[1]{}  

\makeatletter
\renewcommand{\theenumi}{\@roman\c@enumi}
\makeatother

\renewcommand{\phi}{\varphi}
\renewcommand{\epsilon}{\varepsilon}
\renewcommand{\theta}{\vartheta}

\newcommand{\llbracket}{[\negthinspace[}
\newcommand{\rrbracket}{]\negthinspace]}

\def\QQ{{\mathbf Q}}

\def\cO{\mathcal{O}}

\def\fra{\mathfrak{a}}
\def\frb{\mathfrak{b}}

\def\frp{\mathfrak{p}}
\def\frn{\mathfrak{n}}

\def\o{\circ}

\def\.{\cdot}
\def\({\Big{(}}
\def\){\Big{)}}
\def\^{\widehat}
\def\~{\widetilde}
\def\*{{}^*\!}
\def\[{\llbracket}
\def\]{\rrbracket}

\renewcommand{\and}{ \quad \text{and} \quad }

 \DeclareMathOperator{\lct}{lct}
 
 \DeclareMathOperator{\ord}{ord}

\newtheorem{lemma}{Lemma}[section]
\newtheorem{theorem}[lemma]{Theorem}
\newtheorem{corollary}[lemma]{Corollary}
\newtheorem{proposition}[lemma]{Proposition}

\theoremstyle{definition}

\newtheorem{remark}[lemma]{Remark}

\theoremstyle{remark}
\newtheorem*{remark*}{Remark}
\newtheorem*{note*}{Note}

\begin{document}

\title{The Ascending Chain Condition for log canonical thresholds on
l.c.i. varieties}

\author[T.~de Fernex]{Tommaso de Fernex}
\address{Department of Mathematics, University of Utah, 155 South 1400 East,
Salt Lake City, UT 48112-0090, USA} \email{{\tt
defernex@math.utah.edu}}

\author[M. Musta\c{t}\u{a}]{Mircea~Musta\c{t}\u{a}}
\address{Department of Mathematics, University of Michigan,
530 Church Street,
Ann Arbor, MI 48109, USA} \email{{\tt mmustata@umich.edu}}

\begin{abstract}
Shokurov's ACC Conjecture \cite{Sho} says that the set of all log
canonical thresholds
on varieties of bounded dimension satisfies the Ascending Chain Condition.
This conjecture was proved for log canonical
thresholds on smooth varieties in \cite{EM1}. Here we use this result
and inversion of adjunction to establish the conjecture for locally
complete intersection varieties.
\end{abstract}

\thanks{The first author was partially supported by NSF
grant DMS-0548325.
The second author was partially supported by
  NSF grant DMS-0758454, and
  by a Packard Fellowship}
\keywords{Log canonical threshold, Inversion of Adjunction,
locally complete intersection}

\maketitle

\markboth{T.~DE FERNEX AND M.~MUSTA\c T\u A}{LOG CANONICAL THRESHOLDS
ON L.C.I. VARIETIES}

\section{Introduction}

Let $k$ be an algebraically closed field of characteristic zero.
Log canonical varieties
are varieties with mild singularities that provide the most general context
for the Minimal Model Program. More generally, one puts the log canonicity condition
on pairs $(X,\frb^q)$, where $\frb$ is a nonzero ideal on $X$ (most of the times, it is the ideal
of an effective Cartier divisor), and $q$ is a nonnegative real number.
Given a log canonical variety $X$ over $k$, and an ideal sheaf $\frb$ on $X$ with 
$(0)\neq\frb\neq\cO_X$, one defines
the log canonical threshold $\lct(X,\frb)$ of the pair $(X,\frb)$.
This is the largest $q$ such that the pair $(X,\frb^q)$ is log canonical. The log canonical threshold
is a fundamental invariant in birational geometry, see for example \cite{Kol2},
\cite{EM2}, or Chap. 9 in \cite{positivity}.

Shokurov's ACC Conjecture \cite{Sho} says that the set of all log canonical thresholds
on varieties of bounded dimension satisfies the Ascending Chain Condition, that is, it contains no
infinite strictly increasing sequences. This conjecture attracted considerable interest, due to its implications for the Termination of Flips Conjecture (see \cite{Birkar} for a result in this direction).
The first unconditional results on sequences of log canonical thresholds on smooth varieties of arbitrary
dimension have been  obtained in \cite{dFM}, and they were subsequently reproved and strengthened in \cite{Kol1}.
The ACC Conjecture was proved in \cite{EM1} for smooth
ambient varieties, by reducing it to a key special case covered in \cite{Kol1}. In this note we extend this result
to the case when the ambient varieties are locally complete intersection (l.c.i., for short).

\begin{theorem}\label{thm1}
For every $n\geq 1$, the set
$$\{\lct(X,\frb)\mid X\,\text{is l.c.i. and log canonical},\,(0)\neq\frb\neq\cO_X,\,\dim(X)\leq n \}$$
satisfies the Ascending Chain Condition.
\end{theorem}

In fact, we deal with a more general version of log canonical thresholds.
Given a variety $X$ and an ideal sheaf
$\fra$ on $X$ such that the pair $(X,\fra)$ is log canonical, for every ideal
sheaf $\frb$ on $X$  with $(0)\neq\frb\neq\cO_X$, we define the \emph{mixed log canonical
threshold} $\lct_{\fra}(X, \frb)$ to be the largest $q$ such that the pair
$(X,\fra\cdot\frb^q)$ is log canonical. We will prove the following strengthening
of the above theorem.

\begin{theorem}\label{thm2}
For every $n\geq 1$, the set
$$\{\lct_{\fra}(X,\frb)\mid X\,\text{is l.c.i.},\,(X,\fra)\,\text{is log canonical},\,(0)\neq\frb\neq\cO_X,\,
 \dim(X)\leq n\}$$
satisfies the Ascending Chain Condition.
\end{theorem}

A key point for the proof of Theorem~\ref{thm2} is that by Inversion of Adjunction we can express
every number of the form $\lct_{\fra}(X,\frb)$, with $X$ locally complete intersection,
as a similar invariant
on a smooth variety (this is why it is important to work with mixed log canonical thresholds).
Furthermore, we show that if $X$ is an l.c.i. log canonical variety, then $\dim_kT_xX\leq
2\dim(X)$ for every $x\in X$. As a consequence, the above reduction to the smooth case keeps the dimension
of the ambient variety bounded.
After localizing at a suitable point, we reduce ourselves to the case when the ambient space is
${\rm Spec}\,k\llbracket x_1,\ldots,x_{2n}\rrbracket$. In this case we use the framework from
 \cite{dFM} (or equivalently, from \cite{Kol1}) to reduce
 the ACC statement to the case of usual log canonical thresholds on smooth varieties, that was treated in \cite{EM1}.

\subsection{Acnowledgment} We are grateful to Lawrence Ein for several useful discussions.

\section{Mixed log canonical thresholds}

In this section we discuss mixed log canonical thresholds, and show how to reduce
Theorem~\ref{thm2} to the case of ambient smooth varieties. We start by fixing the setup.

Let $k$ be an algebraically closed field of characteristic zero.
In what follows $X$ will be either a variety over $k$ or
${\rm Spec}\, k\llbracket x_1,\ldots,x_n\rrbracket$.
For the basic facts about log canonical pairs in the setting of algebraic varieties, see
\cite{Kol2} or Chap. 9 in \cite{positivity}, while for the case of the spectrum of a formal power series ring
we refer to \cite{dFM}. The key point is that by
\cite{Temkin}, log resolutions exist also in the latter case, and therefore the usual theory
of log canonical pairs carries through.

Suppose that $X$ is as above, and $\fra$, $\frb$  are coherent nonzero sheaves of ideals on $X$ such that
the pair $(X,\fra)$ is log canonical.
In particular, this implies that $X$ is normal
and $\QQ$-Gorenstein. We also assume that $\frb\neq\cO_X$. The \emph{mixed log canonical threshold} of $\frb$ with respect to
$(X,\fra)$ is
$$\lct_{\fra}(X,\frb):={\rm sup}\{q\geq 0\mid (X,\fra\.\frb^q)\, \text{is log canonical}\}.$$

The fact that log canonicity can be checked on a log resolution allows us to describe
this invariant in terms of such a resolution.
Suppose that $\pi\colon Y\to X$ is a log resolution of $\fra\cdot\frb$, and write
$\fra\cdot\cO_Y=\cO(-\sum_ia_iE_i)$, $\frb\cdot\cO_Y=\cO(-\sum_ib_iE_i)$, and
$K_{Y/X}=\sum_ik_iE_i$. It follows from the characterization of log canonicity in terms of
a log resolution that
\begin{equation}\label{eq0}
\lct_{\fra}(X,\frb)=\min\left\{\frac{k_i+1-a_i}{b_i}\mid b_i>0\right\}.
\end{equation}
We see from the above formula that the mixed log canonical threshold is a rational number.
Note also that it is zero if and only if there is $i$ such that $k_i+1=a_i$ and $b_i>0$
(in other words, $(X,\fra)$ is not klt, and there is a non-klt center contained in the support of $\frb$).

\begin{remark}
When $\fra = \cO_X$, the mixed log canonical threshold $\lct_{\cO_X}(X,\frb)$
is nothing else than the usual log canonical threshold $\lct(X,\frb)$,
which we sometimes simply denote by $\lct(\frb)$, when there is no ambiguity regarding
the ambient scheme.
\end{remark}

\begin{remark}\label{rem0}
It follows from the description in terms of a log resolution that if $X=U_1\cup\ldots\cup U_r$, with
$U_i$ open, then
$$\lct_{\fra}(X,\frb)=\min_i\lct_{\fra_i}(U_i,\frb_i),$$
where $\fra_i=\fra\vert_{U_i}$ and $\frb_i=\frb\vert_{U_i}$.
\end{remark}

\begin{remark}\label{rem1}
If $\fra$ and $\frb$ are as above, and $c=\lct_{\fra}(X,\frb)$, then $\lct(X,\fra\cdot \frb^{c})=1$. Indeed, by assumption the pair
$(X,\fra\cdot\frb^c)$ is log canonical, and for every $\alpha>1$ the pair
$(X,(\fra\cdot\frb^c)^{\alpha})$ is not log canonical since
$(X,\fra\cdot\frb^{c\alpha})$ is not.
\end{remark}

\begin{remark}\label{rem2}
Suppose that $X$ is a smooth variety over $k$, and $\fra$, $\frb$ are nonzero ideals on $X$,
with $(X,\fra)$ log canonical and $\frb\neq\cO_X$. If $\lct_{\fra}(X,\frb)=c$, then there is
a (closed) point $x\in X$ such that $\lct_{\fra'}(X',\frb')=c$, where $X'={\rm Spec}(\widehat{\cO_{X,x}})$,
and $\fra'=\fra\cdot\cO_{X'}$, $\frb'=\frb\cdot\cO_{X'}$.
 Indeed, if $\pi\colon Y\to X$ is a log resolution
of $(X,\fra\cdot\frb)$, and if $E_i$ is one of the divisors on $Y$ for which the minimum in
(\ref{eq0}) is achieved, then it is enough to take $x\in\pi(E_i)$ (note that $Y\times_XX'$ gives a log resolution of $(X',\fra'\cdot\frb')$). In this way we are reduced to considering ideals in
$k\llbracket x_1,\ldots,x_n\rrbracket$, where $n=\dim(X)$.
\end{remark}

The following application of
Inversion of Adjunction is a key tool in our study, as it allows us to replace mixed log canonical thresholds on locally complete intersection varieties with the similar type of  invariants on ambient smooth varieties. In this result we assume that our schemes are of finite type over $k$.

\begin{proposition}\label{inversion}
Let $A$ be a smooth irreducible variety over $k$, and $X\subset A$ a closed subvariety of pure codimension $e$, that is normal and locally a complete intersection.
Suppose that $\fra$ and $\frb$ are ideals on $A$, with $\frb\neq\cO_A$, and such that
$X$ is not contained in the union of the zero-loci of $\fra$ and $\frb$.
\begin{enumerate}
\item[i)] The pair $(X,\fra\vert_X)$ is log canonical if and only if
for some open neighborhood $U$ of $X$,  the pair
 $(U,\fra\cdot \frp^e\vert_U)$ is log canonical, where
$\frp$ is the ideal defining $X$ in $A$.
\item[ii)] If $(X,\fra\vert_X)$ is log canonical, and if $X$ intersects the zero-locus  of $\frb$, then for some open neighborhood
$V$ of $X$ we have
$$\lct_{\fra\vert_X}(X,\frb\vert_X)=\lct_{\fra\vert_V\cdot \frp^e\vert_V} (V, \frb\vert_V).$$
\end{enumerate}
\end{proposition}

\begin{proof}
Both assertions follow from Inversion of Adjunction (see Corollary~3.2 in \cite{EM3}), as
this says that
for every nonnegative $q$, the pair $(X,(\fra\.\frb^q)\vert_X)$ is log canonical
if and only if the pair $(A,\fra\.\frb^q\.\frp^e)$ is log canonical in some neighborhood of
$X$.
\end{proof}

The next proposition allows us to control the dimension of the smooth variety,
when replacing a mixed log canonical threshold on an l.c.i. variety by one on a smooth variety.
We keep the assumption that $X$ is of finite type over $k$. Given a closed point $x\in X$,
we denote by $T_xX$ the Zariski tangent space of $X$ at $x$.

\begin{proposition}\label{bound}
Let $X$ be a locally complete intersection variety.
If $X$ is log canonical, then $\dim_kT_xX\leq 2\dim(X)$ for every $x\in X$.
\end{proposition}

\begin{proof}
Fix $x\in X$, and let $N=\dim\,T_xX$. After possibly replacing $X$ by an open neighborhood of $x$,
we may assume that we have a closed embedding of $X$ in a smooth irreducible variety $A$,
of pure codimension $e$, with
$\dim(A)=N$. If $X=A$, then $N=\dim(X)$ and we are done.

Suppose now that $e\geq 1$.
Since $X$ is locally a complete intersection, it follows from Inversion of Adjunction
(see Corollary~3.2 in \cite{EM3}) that the pair $(A,\frp^e)$ is log canonical, where $\frp$ is the ideal
of $X$ in $A$. In particular, if $E$ is the exceptional divisor of
the blow-up $A'$ of $A$ at $p$, and ${\rm ord}_E$ is the corresponding valuation, then
we have
$$N=1+\ord_E(K_{A'/A})\geq e\cdot \ord_E(\frp)\geq 2e=2(N-\dim(X)).$$
This gives $N\leq 2\dim(X)$.
\end{proof}

\begin{corollary}\label{reduction}
In order to prove Theorem~\ref{thm2}, it is enough to show that for every $N\geq 1$, the set
$$\{\lct_{\fra}(W_N,\frb)\mid (W_N,\fra)\,\text{is log canonical},\,(0)\neq\frb\neq\cO_{W_N}\},$$
where $W_N={\rm Spec}\,k\llbracket x_1,\ldots,x_N\rrbracket$,
satisfies the Ascending Chain Condition.
\end{corollary}

\begin{proof}
Let us denote by ${\mathcal M}_N$ the set that appears in the statement. If each ${\mathcal M}_N$
satisfies ACC, then it is clear that in order to prove Theorem~\ref{thm2} it is enough to show
$$\{\lct_{\fra}(X,\frb)\mid X\,\text{is l.c.i.},\,(X,\fra)\,\text{is log canonical},\,\dim(X)\leq n,\,
(0)\neq\frb\neq\cO_X\}
\subseteq\bigcup_{i=1}^{2n}{\mathcal M}_i.$$
Suppose that $(X,\fra)$ is log canonical, with $X$ locally a complete intersection of dimension $\leq n$, and let
$c=\lct_{\fra}(X,\frb)$.
If $x\in X$ is chosen as in Remark~\ref{rem2}, then for every open neighborhood $U$ of $x$
we have $\lct_{\fra\vert_U}(U,\frb\vert_U)=c$. Since $X$ is log canonical, it follows from Proposition~\ref{bound} that $\dim_kT_xX\leq 2n$. After replacing $X$ by a suitable neighborhood of $x$,
we may assume that we have an embedding $X\hookrightarrow A$, where $A$ is a smooth variety of
dimension $m\leq 2n$. Proposition~\ref{inversion} implies that after possibly replacing $A$ by a
neighborhood of $X$, we have $c=\lct_{\fra_1\cdot\frp^{e}}(A,\frb_1)$, where $\frp$
is the ideal defining $X$ in $A$, $e$ is the codimension of $X$ in $A$, and $\fra_1$ and $\frb_1$ are ideals in $A$
whose restrictions to $X$ give, respectively, $\fra$ and $\frb$. We conclude by replacing $A$ with
${\rm Spec}\,k\llbracket x_1,\ldots,x_m\rrbracket$, as in Remark~\ref{rem2}.
\end{proof}

\section{The Ascending Chain Condition}

The proof of Theorem~\ref{thm2} uses a construction from \cite{dFM},
whose general properties we shall review
first. Equivalently, one can replace this construction
by the one considered in \cite{Kol1}, which does not rely on  ultrafilters.

It is shown in \cite{dFM} that we can find an
algebraically closed extension $K$ of $k$, and a way to associate to every
sequence of ideals $\fra_m\subseteq k\llbracket x_1,\ldots,x_n\rrbracket$, an ideal $\widetilde{\fra}$ in $K\llbracket x_1,\ldots,x_n\rrbracket$,
that we shall call {\it the limit ideal} of the sequence $(\fra_m)_m$.
This construction satisfies the following properties:
\begin{enumerate}
\item[i)] If $\fra_m\neq (0)$ for every $m$, then $\widetilde{\fra}\neq (0)$.
\item[ii)] For every sequences of ideals $(\fra_m)_m$ and $(\frb_m)_m$ in
$k\llbracket x_1,\ldots,x_n\rrbracket$, with limit ideals
$\widetilde{\fra}$ and $\widetilde{\frb}$ in $K\llbracket x_1,\ldots,x_n\rrbracket$,
the limit ideals of the sequences $(\fra_m\. \frb_m)_m$ and $(\fra_m+\frb_m)_m$
are, respectively, $\widetilde{\fra}\cdot \widetilde{\frb}$ and $\widetilde{\fra}+
\widetilde{\frb}$.
\item[iii)] If $\fra_m=\fra$ for every $m$, then
$\widetilde{\fra}=\fra\cdot K\llbracket x_1,\ldots,x_n
\rrbracket$.
\item[iv)] If there is an integer $d$ such that each $\fra_m$ can be
generated by polynomials of degree $\leq d$, then the same holds for $\widetilde{\fra}$,
and there are infinitely many $m$ such that
$\lct(\fra_m)=\lct(\widetilde{\fra})$.
\end{enumerate}

The way one applies this construction in \cite{dFM} is the following. 
Suppose that $(\fra_m)_m$ is a sequence
of proper nonzero ideals in $k\llbracket x_1,\ldots,x_n\rrbracket$, and let
$c_m=\lct(\fra_m)$. By properties ii) and iii) above, if we denote
by $\frn$ the maximal ideal in $k\llbracket x_1,\ldots,x_n\rrbracket$, then
the limit ideal $\widetilde{\frn}$ of the constant sequence $(\frn)$
is the maximal ideal in $K\llbracket x_1,\ldots,x_n\rrbracket$, and moreover, for every $d\geq 1$
the limit ideal of the sequence $(\fra_m+\frn^d)_m$ is $\widetilde{\fra}+\widetilde{\frn}^{d}$.
Note also that from properties i) and ii) we deduce that $\widetilde{\fra}$ is a proper nonzero
ideal in $K\llbracket x_1,\ldots,x_n\rrbracket$.
Since each $\fra_m+\frn^{d}$ is generated in degree $\leq d$, it follows from iv)
that we can find infinitely many $m$ such that
$$\lct(\fra_m+\frn^d)=\lct(\widetilde{\fra}+\widetilde{\frn}^d).$$
On the other hand, well-known properties of log canonical thresholds
(see, for example, Corollary 2.11 in \cite{dFM})
give for every $m$ and $d$
$$0\leq \lct(\fra_m+\frn^d)-\lct(\fra_m)\leq \frac{n}{d},$$
$$0\leq \lct(\widetilde{\fra}+\widetilde{\frn}^d)-\lct(\widetilde{\fra})\leq
\frac{n}{d}.$$
It follows that given any $d$, there are infinitely many $m$ such that
$|\lct(\widetilde{\fra})-\lct({\fra_m})|\leq \frac{2n}{d}$. For example,
one deduces from this
that if $\lim_{m\to \infty}\lct(\fra_m)=c$, then $\lct(\widetilde{\fra})=c$.
We refer to \cite{dFM} for details.

For the proof of Theorem~\ref{thm2}, we will need the following proposition.
We point out that while the assertion in i) is an immediate consequence of the above formalism,
the one in ii) is essentially a restatement of the result in \cite{EM1} saying that the ACC Conjecture holds
for log canonical thresholds on smooth varieties.

\vfill\eject

\begin{proposition}\label{prop_ACC}
With the above notation, the following hold:
\begin{enumerate}
\item[i)] If $\lct(\fra_m)\geq \tau$ for every $m\gg 0$, then $\lct(\widetilde{\fra})\geq\tau$.
\item[ii)] There are infinitely many $m$ such that
$\lct(\fra_m)\geq \lct(\widetilde{\fra})$.
\end{enumerate}
\end{proposition}

\begin{proof}
The first assertion follows immediately from the above discussion. Indeed,
since $\lct(\widetilde{\fra})=\lim_{d\to\infty}\lct(\widetilde{\fra}+\widetilde{\frn}^d)$,
 it is enough to
show that $\lct(\widetilde{\fra}+\widetilde{\frn}^d)\geq \tau$ for every $d$. This follows since
given $d$, we can find $m\gg 0$ such that
$$\lct(\widetilde{\fra}+\widetilde{\frn}^d)=\lct(\fra_m+\frn^d)\geq\lct(\fra_m)\geq\tau.$$

For ii), we use the fact that the sequence $(\lct(\fra_m))_m$ contains no strictly increasing infinite
subsequences.  This is a consequence of the fact that by Theorem~1.1 in \cite{EM1}, the set
$${\mathcal T}_n:=\{\lct(\fra)\mid \fra\subseteq k\llbracket x_1,\ldots,x_n\rrbracket, \fra\neq (0), \ord(\fra)\geq 1\}$$
satisfies ACC. In fact, the result in \emph{loc. cit.} concerns log canonical thresholds
of principal ideals in $k[x_1,\ldots,x_n]$, but this implies our assertion by a well-known argument
(see, for example, Proposition 3.6 and Corolary 3.8 in \cite{dFM}).

In particular, we see that there is $\epsilon>0$ such that no $\lct(\fra_m)$ lies
in the open interval $(\lct(\widetilde{\fra})-\epsilon,\lct(\widetilde{\fra}))$. Let $d$ be such that $\frac{2n}{d}<
\epsilon$. We have seen that there are infinitely many $m$ such that
$|\lct(\widetilde{\fra})-\lct(\fra_m)|<\frac{2n}{d}$, and by the choice of $d$, it follows that
for all such $m$ we have $\lct(\fra_m)\geq\lct(\widetilde{\fra})$.
\end{proof}

We can now give the proof of our main result.

\begin{proof}[Proof of Theorem~\ref{thm2}]
In light of Corollary~\ref{reduction}, it is enough to prove that for every $n\geq 1$, if we put
$X={\rm Spec}\,k\llbracket x_1,\ldots,x_n\rrbracket$, then the set
$${\mathcal M}_n:=\{\lct_{\fra}(X,\frb)\mid (X,\fra)\,\text{is log canonical},\,(0)\neq
\frb\neq\cO_X\}$$
satisfies ACC.
Suppose that we have a strictly increasing sequence $(c_m)_m$, with
$c_m=\lct_{\fra_m}(X,\frb_m)$. Let $c=\lim_{m\to\infty}c_m$ (since we have
$\lct_{\fra}(X,\frb)\leq\lct(X,\frb)\leq n$, this sequence is convergent).

We now apply the construction mentioned in the beginning of the section
to get the algebraically closed extension $K$ of $k$
and, for each of the sequences $(\fra_m)_m$ and $(\frb_m)_m$, its limit ideal $\widetilde{\fra}$,
respectively $\widetilde{\frb}$, in $K\llbracket x_1,\ldots,x_n\rrbracket$. We want to compare $c$
with $\widetilde{c}:=\lct_{\widetilde{\fra}}(\widetilde{X},\widetilde{\frb})$, where $\widetilde{X}
={\rm Spec}\,K\llbracket x_1,\ldots,x_n\rrbracket$.
Note that by assumption this mixed log canonical threshold is well-defined, since
we have $\lct(\fra_m)\geq 1$ for every $m$, hence
$\lct(\widetilde{\fra})\geq 1$ by assertion i) in Proposition~\ref{prop_ACC}.

Consider first any positive integers $p$ and $q$ such that $p/q<c$. By assumption, we have
$c_m>p/q$ for all $m\gg 0$. Therefore the pair $(X,\fra_m\.\frb_m^{p/q})$ is log canonical,
 hence
$\lct(\fra_m^q\.\frb_m^p)\geq 1/q$ for all $m\gg 0$. As we have mentioned, the limit ideal in
$K\llbracket x_1,\ldots,x_n\rrbracket$ of the sequence $(\fra_m^q\.\frb_m^p)_m$
is $\widetilde{\fra}^q\.\widetilde{\frb}^p$. Assertion i) in Proposition~\ref{prop_ACC} implies that
$\lct(\widetilde{\fra}^q\.\widetilde{\frb}^p)\geq 1/q$, hence $\widetilde{c}=\lct_{\widetilde{\fra}}
(\widetilde{X},\widetilde{\frb})\geq p/q$. As this holds for every $p/q<c$, we conclude that
$\widetilde{c}\geq c$.

On the other hand,
since $\widetilde{c}\in\QQ$, we may write $\widetilde{c}=\frac{r}{s}$,
for positive integers $r$ and $s$. It follows from Remark~\ref{rem1}
that $\lct(\widetilde{\fra}\.\widetilde{\frb}^{r/s})=1$,
hence $\lct(\widetilde{\fra}^s\.\widetilde{\frb}^r)=1/s$.
Since $\widetilde{\fra}^s\.\widetilde{\frb}^r$ is the limit ideal of the
sequence $(\fra_m^s\.\frb_m^r)_m$, assertion ii) in
Proposition~\ref{prop_ACC} implies that there are infinitely many $m$ such that
$\lct(\fra_m^s\.\frb_m^r)\geq 1/s$, hence $\lct_{\fra_m}(X,\frb_m)\geq r/s$.
For such $m$ we have
$$\widetilde{c}\geq c>c_m\geq \frac{r}{s}=\widetilde{c},$$
a contradiction. This completes the proof of the theorem.
\end{proof}

\providecommand{\bysame}{\leavevmode \hbox \o3em
{\hrulefill}\thinspace}

\end{document}